\documentclass[10pt]{article}
\usepackage{pifont}
\usepackage{amsfonts}
\usepackage{amsmath,amsthm,amssymb}
\usepackage{amsmath, epsfig}
\usepackage{latexsym}
\usepackage{amssymb}
\usepackage{mathrsfs}
\usepackage{amscd}
\usepackage{xypic}
\usepackage[all]{xy}
\usepackage{dsfont}


 
\newtheorem{thm}{Theorem}[section]

\newtheorem{lem}[thm]{Lemma}

\newtheorem{defn}[thm]{Definition}

\newtheorem{prop}[thm]{Proposition}

\begin{document}
\title{\bf Blocks with abelian defect groups of rank $2$ and one simple module}
\author{Xueqin Hu}
\date{\small
School of Mathematics and Statistics, Central China Normal University,
Wuhan 430079, China}
\maketitle

\begin{abstract}

In this paper, we investigate the block that has an abelian defect group of rank $2$
and its Brauer correspondent has only one simple module.
We will get an isotypy between the block and its Brauer correspondent.
It will generalize the result of Kessar and Linckelmann (\cite{KL}).
\end{abstract}
\section{Introduction}
Let $p$ be a prime and
$\mathcal{O}$ a complete discrete valuation ring
having an algebraically closed residule field $k$ of characteristic $p$
and a quotient field $\mathcal{K}$ of characteristic $0$.
We will always assume that $\mathcal{K}$ is big enough for the finite groups below.

Let $G$ be a finite group and $b$ a block of $\mathcal{O}G$
with a defect group $P$.
Denote by $\mathrm{Irr}_\mathcal{K}(G,b)$ and
$\mathrm{IBr}(G,b)$
the set of irreducible ordinary characters in $b$
and the set of irreducible Brauer characters in $b$ respectively.
Set $l_G(b)=|\mathrm{IBr}(G,b)|$.
Let $c$ be the Brauer correspondent of $b$ in $N_G(P)$.
In \cite{KL},
Kessar and Linckelmann investigated the block $b$
under the assumptions that
$l_{N_G(P)}(c)=1$ and $P$ is elementary abelian of rank $2$.
They showed that
the inertial quotient of $b$ is abelian
and there is an isotypy between $b$ and $c$
all of whose signs are positive.

In this note,
we will generalize these results to the blocks with defect groups of rank $2$.

\begin{thm}\label{MT}
  Keep the notation as above.
  Assume that $P$ is abelian of rank $2$
  and $l_{N_G(P)}(c)=1$.
  Then the inertial quotient of $b$ is abelian
  and
  there is an isotypy between $b$ and $c$.
\end{thm}

These results are well-known when either $p$ is $2$ or
the inertial quotient of $b$ is trivial.
Therefore,
we may assume that $p$ is odd
and
the inertial quotient of $b$ is non-trivial throughout this paper.

\section{The structure of the block $c$}
Keep the notation as above.
In this section,
we will investigate the structure of the inertial quotient of $b$
and irreducible ordinary characters of the block $c$.

Given a positive integer $a$,
denote by $C_a$ the cyclic group of order $a$.
We will use $[-\,\ ,\,\ -]$ to represent the commutator.
Assume that $P=C_{p^n}\times C_{p^m}$ for some positive integers $n,m$.
We will fix a maximal $b$-Brauer pair $(P,b_P)$.
For any $Q\leq P$,
denote by $(Q,b_Q)$ the unique $b$-Brauer pair contained in $(P,b_P)$.
Let $E$ be the inertial quotient of $b$ associated with $(P,b_P)$,
namely,
$E=N_G(P,b_P)/C_G(P)$.

\begin{lem}\label{E is abelian}
  The inertial quotient $E$ is abelian if $l_{N_G(P)}(c)=1$.
\end{lem}

\begin{proof}
  Let $\Phi(P)$ be the Frattini subgroup of $P$.
  So $P/\Phi(P)$ is $C_p\times C_p$.
  Set $H$ to be $N_G(P,b_P)$.
  Then $\Phi(P)\unlhd H$
  and denote $H/\Phi(P)$ by $\bar{H}$.
  For any subset $X$ of $\mathcal{O}H$,
  $\bar{X}$ denotes the image of $X$ under the canonical map
  $\mathcal{O}H\longrightarrow\mathcal{O}\bar{H}$.

  Since $l_H(b_P)=1$,
  $l_{\bar{H}}(\bar{b}_P)=1$
  and
  $\bar{b}_P$ is a block of $\bar{H}$ with defect group
  $\bar{P}=C_p\times C_p$.
  Let $\hat{C}$ be the subgroup of $H$
  such that
  $\hat{C}/\Phi(P)=C_{\bar{H}}(\bar{P})$.
  Hence,
  $\hat{C}=\{x\in H~|~[P,x]\subseteq\Phi(P)\}$.
  It is clear that
  $P=[P,\hat{C}]\times C_P(\hat{C})$.
  So $P=C_P(\hat{C})$ since
  $[P,\hat{C}]\leq\Phi(P)$.
  This means $C_{\bar{H}}(\bar{P})=\bar{C}_G(P)$.
  Hence,
  $(\bar{P},\bar{b}_P)$ is a maximal $\bar{b}_P$-Brauer pair
  of $\mathcal{O}\bar{H}\bar{b}_P$.
  By \cite[Proposition 5.2]{KL},
  $N_{\bar{H}}(\bar{P},\bar{b}_P)/C_{\bar{H}}(\bar{P})$ is abelian.
  It is evident that $E$ is isomorphic to $N_{\bar{H}}(\bar{P},\bar{b}_P)/C_{\bar{H}}(\bar{P})$.
  We are done.
\end{proof}

By \cite[Lemma 2]{DJ} and
the structure of blocks with normal defect groups,
$E$ is a direct product of two isomorphic groups.
Next,
we will show that $E$ acts diagonally on $P$.
This can be deduced from the following general fact.

\begin{lem}\label{diagonal action}
  Let $D$ be an abelian $p$-group of rank $2$
  and
  $F\leq\mathrm{Aut}(P)$ an abelian $p^\prime$-group
  which is a direct product of two isomorphic subgroups.
  Then we have the decompositions
  $F=F_1\times F_2$ and $D=D_1\times D_2$ such that
  $F_1$ acts faithfully on $D_1$ and centralises $D_2$
  and
  $F_2$ acts faithfully on $D_2$ and centralises $D_1$
  and
  $F_1\cong F_2$.
  In particular,
  $F_1$ and $F_2$ are cyclic groups of order dividing $(p-1)$.
\end{lem}

\begin{proof}
  We will exhibit it by induction on $|D|$.
  When $D$ is elementary abelian,
  it is actually done in \cite[Proposition 5.3]{KL}.
  We may assume that $n\geq 2$ or $m\geq 2$.
    Let $\Phi(D)$ be the Frattini subgroup of $D$.
  So $D/\Phi(D)$ is $C_p\times C_p$.
  Let $\pi$ be the canonical map
  from $F$ to $\mathrm{Aut}(D/\Phi(D))$.
  For any subset $X$ of $F$,
  $\bar{X}$ denotes the image of $X$ under $\pi$.
   It is clear that $\pi$ is injective.
  So there exist
  two subgroups $F_1$ and $F_2$ of $F$ and
  two subgroups $D_1$ and $D_2$ of $D$ containing $\Phi(D)$
  satisfying the properties
  $\bar{F}=\bar{F}_1\times \bar{F}_2$ and
  $D/\Phi(D)=D_1/\Phi(D)\times D_2/\Phi(D)$ and
  $\bar{F}_1$ acts faithfully on $D_1/\Phi(D)$ and centralises $D_2/\Phi(D)$
  and
  $\bar{F}_2$ acts faithfully on $D_2/\Phi(D)$ and centralises $D_1/\Phi(D)$
  and
  $\bar{F}_1\cong \bar{F}_2$.
  Hence,
  $D_1$ and $D_2$ are $F$-stable
  and
  they fulfill

  (i) $D_1=[D_1,F_1]\cdot\Phi(D)$ and $[D_1,F_2]\subseteq\Phi(D)$
      and $F_1$ acts faithfully on $D_1$;

  (ii) $D_2=[D_2,F_2]\cdot\Phi(D)$ and $[D_2,F_1]\subseteq\Phi(D)$
      and $F_2$ acts faithfully on $D_2$;

  (iii) $D_1\cap D_2=\Phi(D)$ and $D_1/\Phi(D)\cong C_p\cong D_2/\Phi(D)$
      and $D=D_1\cdot D_2$.
  \vspace{0.5mm}

  Suppose that $\Phi(D)$ is cyclic.
  Then $D=C_p\times C_{p^m}$ with $m\geq 2$
  and $\Phi(D)=C_{p^m-1}$.
  Since $D_2=[D_2,F_1]\times C_{D_2}(F_1)$ and
  $[D_2,F_1]\subseteq\Phi(D)$,
  $\Phi(D)=[D_2,F_1]\times C_{\Phi(D)}(F_1)$.
  Then
  either $[D_2,F_1]=1$
  or
  $C_{\Phi(D)}(F_1)=1$
  by the assumption that $\Phi(D)$ is cyclic.
  If $[D_2,F_1]=1$,
  then
  $\Phi(D)\leq D_2\leq C_D(F_1)$.
  Clearly,
  $D=[D,F_1]\times C_D(F_1)$
  and
  $D_2$ is a maximal subgroup of $D$.
  Thus,
  $D_2=C_D(F_1)$ and
  $[D,F_1]=[D_1,F_1]$.
  Since $F_1$ and $F_2$ commute with each other
  and
  $[D_1,F_2]\subseteq\Phi(D)\subseteq C_D(F_1)$,
  $[[D_1,F_1],F_2]=1$.
  So
  $[D_1,F_1]\leq C_P(F_2)$.
  Since $F_1$ acts faithfully on $D_1$
  and $D_1=[D_1,F_1]\times C_{D_1}(F_1)$,
  $F_1$ acts faithfully on $[D_1,F_1]$.
  Thus,
  the decompositions
  $F=F_1\times F_2$
  and
  $D=[D_1,F_1]\times D_2$
  are what we want.
  We may assume that
  $C_{\Phi(D)}(F_1)=1$.
  Then
  $\Phi(D)=[D_2,F_1]$ and
  $D_2=\Phi(D)\times C_{D_2}(F_1)$.
  If $C_{\Phi(D)}(F_2)=1$,
  we can get
  $\Phi(D)=[D_1,F_2]$
  and
  $D_1=\Phi(D)\times C_{D_1}(F_2)$ similarly.
  Then
  $D=C_{D_1}(F_2)\times C_{D_2}(F_1)$ which is impossible.
  So $C_{\Phi(D)}(F_2)\neq1$.
  Then replacing $D_2$ by $D_1$ in the previous argument,
  we can obtain the decompositions that we need.

  Suppose $\Phi(D)$ is of rank $2$.
  Then both $D_1$ and $D_2$ are of rank $2$.
  Let $K$ be subgroup of $F$ consisting of automorphisms
  acting trivially on $D_1$.
  Then
  $D=[D,K]\times C_D(K)$
  and
  $D_1\leq C_D(K)$.
  Hence,
  $K$ has to be trivial
  since $D_1$ has rank $2$.
  This means $F$ acts faithfully on $D_1$.
  By induction,
  we have
  $D_1=D_{11}\times D_{12}$ and
  $F=F_{11}\times F_{12}$ such that
  $F_{11}$ acts faithfully on $D_{11}$ and
  centralises $D_{12}$
  and
  $F_{12}$ acts faithfully on $D_{12}$ and
  centralises $D_{11}$
  and
  $F_{11}\cong F_{12}$.
  Then
  $D=[D,F_{11}]\times C_{D}(F_{11})$
  and
  $D_{11}=[D_{11},F_{11}]\leq[D,F_{11}]$
  and
  $D_{11}\leq C_{D}(F_{12})$.
  In particular,
  $C_{[D,F_{11}]}(F_{12})\neq 1$.
  But
  $[D,F_{11}]$ is cyclic.
  Then
  $[D,F_{11}]\leq C_D(F_{12})$
  and moreover
  $C_D(F_{12})=[D,F_{11}]\times(C_D(F_{11})\cap C_D(F_{12}))$.
  But
  $C_D(F_{12})$ is also cyclic.
  We have
  $[D,F_{11}]=C_D(F_{12})$.
  Similarly,
  we can prove that
  $[D,F_{12}]=C_D(F_{11})$.
  Then the decompositions
  $D=[D,F_{11}]\times [D,F_{12}]$
  and
  $F=F_{11}\times F_{12}$
  are what we want.
  We are done.
\end{proof}

Hence,
by Lemma \ref{diagonal action},
we have
$E=E_1\times E_2$ and
$P=P_1\times P_2$ such that

(i) $E_1$ acts faithfully on $P_1$ and centralises $P_2$;

(ii) $E_2$ acts faithfully on $P_1$ and centralises $P_1$;

(iii) $E_1\cong E_2$ are cyclic groups of order $l$,
which $l$ is a positive integer dividing $(p-1)$.

We can easily describe the source algebra of the block $c$
by the structure theory of blocks with normal defect groups
and the structure  of inertial quotient $E$.
It is well-known that there exists a central extension

$$\xymatrix@C=0.5cm{
  1 \ar[r] & Z \ar[r] & \tilde{E} \ar[r] & E \ar[r] & 1 }$$
  with
  $Z$ cyclic $p^\prime$-group
  such that
  there is an irreducible ordinary character $\theta$ of $Z$
  which is covered by a unique irreducible character of $\tilde{E}$.
  Let $e_\theta\in\mathcal{O}Z$ be the central idempotent
  corresponding to $\theta$.
  Set $N=P\rtimes\tilde{E}$.
  Then $\mathcal{O}Ne_\theta$ is the source algebra of the block $c$.
  Note that $e_\theta$ is still a block of $C_N(R)$ for any $R\leq P$.
  The following lemma gives some information about the degrees and number
  of irreducible ordinary characters of $\mathcal{O}Ne_\theta$,
  which is similar with \cite[Proposition 5.3]{KL}.
  We will skip the proof.

  \begin{lem}\label{irreducible character of c}
    Set $A$ to be $\mathcal{O}Ne_\theta$.
    Then the degree of an element of $\mathrm{Irr}_\mathcal{K}(A)$
    is either $l$ or $l^2$ and
    $\mathrm{Irr}_\mathcal{K}(A)$ has
    $p^n+p^m-1$ elements of degree $l$
    and
    $\frac{p^n-1}{l}\cdot\frac{p^m-1}{l}$
    elements of degree $l^2$.
  \end{lem}

\section{The extension of local system}
Keep the notation as above.
In this section,
we will use the so-called $(G,b)$-local system
introduced by Puig and Usami in \cite{PU}
to prove the main theorem.

First,
let us recall some notation and
state the definition of $(G,b)$-local system
under our setting
(see \cite{PU}).

Let $\mathcal{CF}_\mathcal{K}(G)$ be the
vector space of $\mathcal{K}$-valued class functions of $G$
and
$\mathcal{BCF}_\mathcal{K}(G)$ be the
vector space of $\mathcal{K}$-valued class functions on the set
$G_{p^\prime}$ of $p^\prime$-elements of $G$.
It is clear that
the set of irreducible ordinary characters of $G$ is a $\mathcal{K}$-basis
of $\mathcal{CF}_\mathcal{K}(G)$
and
the set of irreducible Brauer characters of $G$ is a $\mathcal{K}$-basis
of $\mathcal{BCF}_\mathcal{K}(G)$.
For $\chi,\chi^\prime\in\mathcal{CF}_\mathcal{K}(G)$,
we denote by
$\langle\chi,\chi^\prime\rangle$ the inner product of $\chi$ and $\chi^\prime$.

Let $u$ be a $p$-element of $G$.
we have the well-known surjective $\mathcal{K}$-linear map
$d_G^u:\mathcal{CF}_\mathcal{K}(G)\longrightarrow
\mathcal{BCF}_\mathcal{K}(C_G(u))$
defined by
$d_G^u(\chi)(s)=\chi(us)$
for any
$\chi\in\mathcal{CF}_{\mathcal{K}}(G)$
and
$s\in C_G(u)_{p^\prime}$.
It has a section
$e_G^u:\mathcal{BCF}_\mathcal{K}(C_G(u))\longrightarrow
\mathcal{CF}_\mathcal{K}(G)$
such that
for $\varphi\in\mathcal{BCF}_\mathcal{K}(C_G(u))$,
$e_G^u(\varphi)(g)=0$ if
the $p$-part of $g$ is not conjugate to $u$ in $G$.

For the block $b$,
let $\mathcal{CF}_\mathcal{K}(G,b)$
be the subspace of $\mathcal{CF}_\mathcal{K}(G)$
generated by the elements in
$\mathrm{Irr}_\mathcal{K}(G,b)$
and
$\mathcal{L}_\mathcal{K}(G,b)$
the group of generalized characters in $b$.
Also,
let
$\mathcal{CF}_\mathcal{K}^\circ(G,b)=
\mathcal{CF}_\mathcal{K}(G,b)\cap\mathrm{Ker}(d_G^1)$
and
$\mathcal{L}_\mathcal{K}^\circ(G,b)=
\mathcal{L}_\mathcal{K}(G,b)\cap\mathrm{Ker}(d_G^1)$.

\begin{defn}(Puig-Usami \cite[3.2]{PU})
  With the above notation and assumption.
  Let $X$ be an $E$-stable non-empty set of subgroups of $P$
  and assume that $X$ contains any subgroup of $P$ containing an element of $X$.
  Let $\Gamma$ be a map over $X$ sending $Q\in X$ to a bijective isometry
  $$\Gamma_Q:\mathcal{BCF}_{\mathcal{K}}(C_N(Q),e_\theta)\longrightarrow
  \mathcal{BCF}_\mathcal{K}(C_G(Q),b_Q).$$
  If $\Gamma$ satisfies the following conditions,
  then $\Gamma$ is called a $(G,b)$-{\rm{local system}} over $X$.

  (i) For any $Q\in X$,
  any $\eta\in\mathcal{BCF}_\mathcal{K}(C_N(Q),e_\theta)$
  and
  any $s\in E$,
  we have
  $\Gamma_Q(\eta)^s=\Gamma_{Q^s}(\eta^s)$.

  (ii) For any $Q\in X$
  and any $\eta\in\mathcal{L}_\mathcal{K}(C_N(Q),e_\theta)$,
  the sum
  $$\sum\limits_{u}e_{C_G(Q)}^u(\Gamma_{Q\cdot\langle u\rangle}
  (d_{C_N(Q)}^u(\eta)))$$
  where $u$ runs over a set of representatives $U_Q$ for the orbits
  of $C_E(Q)$ in $P$,
  is a generalized character of $C_G(Q)$.
\end{defn}

Let $\Gamma$ be a $(G,b)$-local system over $X$.
Such $\Gamma$ always exists by \cite[3.4.2]{PU}.
For any $Q\in X$,
we have a map
$\Delta_Q:
\mathcal{CF}_\mathcal{K}(C_N(Q),e_\theta)\longrightarrow
\mathcal{CF}_\mathcal{K}(C_G(Q),b_Q)$ defined by
$$\Delta_Q(\eta)=
\sum\limits_{u\in U_Q}e_{C_G(Q)}^u(\Gamma_{Q\cdot\langle u\rangle}
  (d_{C_N(Q)}^u(\eta))).$$
Then by \cite[3.3 and 3.4]{PU}
$\Delta_Q$ gives a perfect isometry between
the block $e_\theta$ of $C_N(Q)$
and
the block $b_Q$ of $C_G(Q)$
and
$\Delta_Q(\lambda\ast\eta)=\lambda\ast\Delta_Q(\eta)$
for any
$\lambda\in\mathcal{CF}_\mathcal{K}(P)^{C_E(Q)}$
and
$\eta\in\mathcal{CF}_\mathcal{K}(C_N(Q))$.
Here,
$\mathcal{CF}_\mathcal{K}(P)^{C_E(Q)}$ denotes the set
of $C_E(Q)$-stable elements of $\mathcal{CF}_\mathcal{K}(P)$
and
$\ast$ denotes the $\ast$-construction of charaters
due to Brou$\acute{\mathrm{e}}$ and Puig (see \cite{BP}).
Hence,
if $X$ contains the trivial subgroup $1$ of $P$,
then $\Delta_1$ induces a perfect isometry
between the block $e_\theta$ of $N$
and
the block $b$ of $G$.
Moreover,
this is an isotypy in the sense of \cite{B}
by \cite[Proposition 2.7]{WZZ}.

In \cite{PU},
Puig and Usami developed a criterion for the
extendibility of the $(G,b)$-local system.
With the notation above.
Suppose that $1\not\in X$
and
let $Q$ be a maximal subgroup of $P$
such that $Q\not\in X$.
Denote by $X^\prime$ the union of
$X$ and the $E$-orbit of $Q$.
For any subset $Y$ of $\mathcal{O}C_N(Q)$,
denote by $\bar{Y}$ the image of $Y$
under the canonical map from
$\mathcal{O}C_N(Q)$ to
$\mathcal{O}C_N(Q)/Q$.
We have the similar notation for $\mathcal{O}C_G(Q)$.
So $\bar{e}_\theta$ and $\bar{b}_Q$ are the blocks of
$\bar{C}_N(Q)$ and $\bar{C}_G(Q)$ respectively.
Set
$\Delta_Q^\circ=
\sum\limits_{u\in U_Q-Q}e_{C_G(Q)}^u\circ\Gamma_{Q\cdot\langle u\rangle}
  \circ d_{C_N(Q)}^u$
(see \cite[3.6.2]{PU}).
By \cite[Proposition 3.7 and Remark 3.8]{PU},
$\Delta_Q^\circ$ induces a bijective isometry
$$\bar{\Delta}_Q^\circ:
\mathcal{CF}_\mathcal{K}^\circ(\bar{C}_N(Q),\bar{e}_\theta)\cong
\mathcal{CF}_\mathcal{K}^\circ(\bar{C}_G(Q),\bar{b}_Q)$$
such that
$\bar{\Delta}_Q^\circ(\mathcal{L}_\mathcal{K}^\circ(\bar{C}_N(Q),\bar{e}_\theta))=
\mathcal{L}_\mathcal{K}^\circ(\bar{C}_G(Q),\bar{b}_Q)$.
Clearly,
$\bar{\Delta}_Q^\circ(\lambda\ast\eta)=
\lambda\ast\bar{\Delta}_Q^\circ(\eta)$
for $\lambda\in\mathrm{Irr}_\mathcal{K}(\bar{P})^{C_E(Q)}$
and
$\eta\in\mathcal{L}_\mathcal{K}^\circ(\bar{C}_N(Q),\bar{e}_\theta)$
(see \cite[Case 2.2]{W05}) and
$\bar{\Delta}_Q^\circ$ is $N_E(Q)$-stable.
The following is the key criterion of extendibility.

\begin{prop}(\cite[Proposition 3.11]{PU})\label{extendibility}
 With the notation above,
 the $(G,b)$-local system $\Gamma$ over $X$ can be extended to a
 $(G,b)$-local system $\Gamma^\prime$ over $X^\prime$
 if and only if
 $\bar{\Delta}_Q^\circ$ can be extended to an $N_E(Q)$-stable bijective isometry
 $$\bar{\Delta}_Q:\mathcal{CF}_\mathcal{K}(\bar{C}_N(Q),\bar{e}_{\theta})\cong
 \mathcal{CF}_\mathcal{K}(\bar{C}_G(Q),\bar{b}_Q)$$
 such that
 $\bar{\Delta}_Q(\mathcal{L}_\mathcal{K}(\bar{C}_N(Q),\bar{e}_\theta))=
\mathcal{L}_\mathcal{K}(\bar{C}_G(Q),\bar{b}_Q)$.
\end{prop}

In order to prove Theorem \ref{MT},
it suffices to show that
there is a $(G,b)$-local system over the set of all the subgroups $P$.
Hence,
by Proposition \ref{extendibility},
we can assume that
there is a $(G,b)$-local system $\Gamma$ over $X$
such that
$1\not\in X$
and
$Q$ is a maximal subgroup of $P$
such that $Q\not\in X$.

\begin{thm}\label{MT'}
With the notation above and
assumptions of Section $2$.
Then $\bar{\Delta}_Q^\circ$ can be extended to an $N_E(Q)$-stable bijective isometry
 $$\bar{\Delta}_Q:\mathcal{CF}_\mathcal{K}(\bar{C}_N(Q),\bar{e}_{\theta})\cong
 \mathcal{CF}_\mathcal{K}(\bar{C}_G(Q),\bar{b}_Q)$$
 such that
 $\bar{\Delta}_Q(\mathcal{L}_\mathcal{K}(\bar{C}_N(Q),\bar{e}_\theta))=
\mathcal{L}_\mathcal{K}(\bar{C}_G(Q),\bar{b}_Q)$.
\end{thm}

\begin{proof}
  By the structure of $E$ and $P$,
  $C_E(Q)$ has only three possibilities: $1$, $E$ and
  $E_1$ or $E_2$.
  So we will divided the proof into $3$ cases.

  \vspace{5mm}
  {\bf Case 1}~~~Assume that $C_E(Q)=1$.

  Then the blocks $e_\theta$ of $C_N(Q)$ and
  $b_Q$ of $C_G(Q)$ are nilpotent.
  By the same argument as in \cite[4.4]{PU},
  $\bar{\Delta}_Q^\circ$ can be extended to an $N_E(Q)$-stable bijective isometry
  $\bar{\Delta}_Q$.

    \vspace{5mm}
  {\bf Case 2}~~~Assume that $C_E(Q)=E$.

  Then $Q$ has to be trivial subgroup of $P$ and
  $N_E(Q)=E$.
  So $\bar{C}_N(Q)=N$ and
  $\bar{C}_G(Q)=G$ and
  we have a bijective isometry
  $$\bar{\Delta}^\circ:\mathcal{CF}_\mathcal{K}^\circ(N,e_\theta)\longrightarrow
  \mathcal{CF}_\mathcal{K}^\circ(G,b)$$
  such that
  $\bar{\Delta}^\circ(\mathcal{L}_\mathcal{K}^\circ(N,e_\theta))=
  \mathcal{L}_\mathcal{K}^\circ(G,b)$.

  The following technique we adopt to extend $\bar{\Delta}^\circ$ is essentially due to
  Kessar and Linckelmann (see \cite[Theorem 4.1]{KL}).

  By Lemma \ref{irreducible character of c},
  we have the following disjoint union
  $$\mathrm{Irr}_\mathcal{K}(N,e_\theta)=\Lambda_1\cup\Lambda_2,$$
  where
  $\Lambda_1$ consists of irreducible ordinary characters of dimension $l$
  and
  $\Lambda_2$ consists of irreducible ordinary characters of dimension $l^2$.
  Hence,
  $|\Lambda_1|=p^n+p^m-1$
  and
  $|\Lambda_2|=\frac{p^n-1}{l}\cdot\frac{p^m-1}{l}$.
  We can assume that $n\geq 2$.
  Then $|\Lambda_1|>2$ and $|\Lambda_2|>2$.
  Choose an element $\psi_i\in\Lambda_i$ and
  set
  $\Lambda_i^\prime=\Lambda_i-\{\psi_i\}$
  for $i=1,2$.
  Since $l_{N}(e_\theta)=1$,
  it is easy to see
  $$\mathcal{B}=\{\psi_1-\psi_1^\prime\,|\,\psi_1^\prime\in\Lambda_1^\prime\}
  \cup\{\psi_2-\psi_2^\prime\,|\,\psi_2^\prime\in\Lambda_2^\prime\}
  \cup\{\psi_2-l\psi_1\}$$
  is a $\mathbb{Z}$-basis of $\mathcal{L}_\mathcal{K}^\circ(N,e_\theta)$.
  Since $p$ is odd, $|\Lambda_i^\prime|\geq 3$ for $i=1,2$.
  So by the same argument in \cite[4.4]{PU},
  for any $i=1,2$,
  there exists a subset $\Omega_i=\{\chi_{\psi_i},\chi_{\psi_i^\prime}\,|\,
  \psi_i^\prime\in\Lambda_i^\prime\}$ of
  $\mathrm{Irr}_\mathcal{K}(G,b)$
  and
  $\delta_i\in\{\pm 1\}$
  such that
  $\bar{\Delta}^\circ(\psi_i-\psi_i^\prime)=
  \delta_i(\chi_{\psi_i}-\chi_{\psi_i^\prime})$.
  Since
  $\langle\psi_1-\psi_1^\prime,
  \psi_2-\psi_2^\prime\rangle=0$
  for any
  $\psi_1^\prime\in\Lambda_1^\prime$
  and
  $\psi_2^\prime\in\Lambda_2^\prime$,
  $\{\chi_{\psi_1},\chi_{\psi_1^\prime}\,|\,
  \psi_1^\prime\in\Lambda_1^\prime\}$
  and
  $\{\chi_{\psi_2},\chi_{\psi_2^\prime}\,|\,
  \psi_2^\prime\in\Lambda_2^\prime\}$
  have trivial intersection.
  Denote $\psi_2-l\psi_1$ by $\mu$.
  Then
  $\langle\mu,
  \psi_1-\psi_1^\prime\rangle=-l$
  for all $\psi_1^\prime\in\Lambda_1^\prime$.
  Thus
  \begin{equation}\label{equation E}
    \begin{array}{ll}
    \bar{\Delta}^\circ(\mu)=
    \delta_1(a-l)\chi_{\psi_1}+
    \delta_1a\sum\limits_{\psi_1^\prime\in\Lambda_1^\prime}\chi_{\psi_1^\prime}+
    \Xi
    \end{array}
  \end{equation}
  for some integer $a$ and
  some element $\Xi\in\mathcal{L}_\mathcal{K}(N,e_\theta)$
  not involving any of elements in $\Omega_1$.
  Since
  $\langle\mu,
  \psi_2-\psi_2^\prime\rangle=1$
  and
  $\bar{\Delta}^\circ(\psi_2-\psi_2^\prime)=
  \delta_2(\chi_{\psi_2}-\chi_{\psi_2^\prime})$,
  $\Xi$ must involve one of the two characters occuring in
  $\bar{\Delta}^\circ(\psi_2-\psi_2^\prime)$
  for any $\psi_2^\prime\in\Lambda_2^\prime$.
  Taking norms on both sides in equation (\ref{equation E}),
  we have
  \begin{equation}\label{inequation norm}
    \begin{array}{ll}
      &1+l^2\geq(a-l)^2+(p^n+p^m-2)a^2=
      (p^n+p^m-1)a^2-2la+l^2\\
      \Longleftrightarrow
      &1\geq(p^n+p^m-1)a^2-2la
    \end{array}
  \end{equation}

  Suppose that $a\leq0$.
  Since $a$ is integer and
  $p^n+p^m-1,l$ are positive integers,
  $a$ has to be $0$.

  Suppose that $a>0$.
  Since $p^n+p^m-1>2l$,
  $(p^n+p^m-1)a^2-2la>(p^n+p^m-1)(a^2-a)$.
  This forces $a=1$.
  Hence,
  $a=0$ or $1$.
  Notice that $\Xi\neq 0$.
  This implies (\ref{inequation norm}) is a proper inequality.
  So $a$ must be $0$.
  Then equation (\ref{equation E}) becomes
  $$\bar{\Delta}^\circ(\mu)=-\delta_1l\chi_{\psi_1}+\Xi.$$
  Comparing norms,
  we have
  $\langle\Xi,\Xi\rangle=1$.

  For any $\psi_2^\prime,\psi_2^{\prime\prime}\in\Lambda_2^\prime$,
  $$\langle\bar{\Delta}^\circ(\mu),
  \delta_2(\chi_{\psi_2}-\chi_{\psi_2^\prime})\rangle=
  \langle\mu,
  \psi_2-\psi_2^\prime\rangle=1$$
  and
  $$\langle\bar{\Delta}^\circ(\mu),
  \delta_2(\chi_{\psi_2^{\prime}}-\chi_{\psi_2^{\prime\prime}})\rangle=
  \langle\mu,
  \psi_2^\prime-\psi_2^{\prime\prime}\rangle=0.$$
  Then $\Xi=\delta_2\chi_{\psi_2}$.
  But $\bar{\Delta}^\circ(\mu)(1)=0$.
  This forces $\delta_1=\delta_2$.
  Since $\mathcal{B}$ is a $\mathbb{Z}$-basis of
  $\mathcal{L}_\mathcal{K}^\circ(N,e_\theta)$,
  $\mathrm{Irr}_\mathcal{K}(G,b)=\Omega_1\cup\Omega_2$.
  Hence,
  we get a bijective isometry $\bar{\Delta}$ from
  $\mathcal{L}_\mathcal{K}(N,e_\theta)$
  to
  $\mathcal{L}_\mathcal{K}(G,b)$
  mapping $\psi_i$ and $\psi_i^\prime$
  to $\chi_{\psi_i}$ and $\chi_{\psi_i^\prime}$
  respectively,
  where $i=1,2$.
  In particular,
  $l_G(b)=l_N(e_\theta)=1$.
  Clearly, it is an extension of $\bar{\Delta}^\circ$.
  Since $\bar{\Delta}^\circ$ is $E$-stable
  and $l_G(b)=l_N(e_\theta)=1$,
  $\bar{\Delta}$ is also $E$-stable.


      \vspace{5mm}
  {\bf Case 3}~~~Assume that $C_E(Q)=E_i$ for some $i=1,2$.

  We can assume that $C_E(Q)=E_1$ and then $1\neq Q\leq P_2$
  and $N_E(Q)=E$.
  It suffices to prove that
  $\bar{\Delta}_Q^\circ$ can extend to an $E_2$-stable bijective isometry
  $\bar{\Delta}_Q:\mathcal{L}_\mathcal{K}(\bar{C}_N(Q),\bar{e}_\theta)\longrightarrow
  \mathcal{L}_\mathcal{K}(\bar{C}_G(Q),\bar{b}_Q)$.

    By \cite[Theorem 1]{W14},
  $|\mathrm{Irr}_\mathcal{K}(\bar{C}_N(Q),\bar{e}_\theta)|=
  |\mathrm{Irr}_\mathcal{K}(\bar{C}_G(Q),\bar{b}_Q)|$
  and
  $l_{\bar{C}_N(Q)}(\bar{e}_\theta)=l_{\bar{C}_G(Q)}(\bar{b}_Q)$
  since the block $\bar{b}_Q$ of $\bar{C}_G(Q)$ has a cyclic hyperfocal subgroup.
  It is clear $C_N(Q)=(P_1\rtimes\tilde{E}_1)\times P_2$
  and $C_N(Q)\unlhd N$,
  where $\tilde{E}_1$ is the preimage of $E_1$ in $\tilde{E}$.
  Hence,
  $E_1$ is the inertial quotient of the block $e_\theta$ of $C_N(Q)$
  and
  $P_1$ is a hyperfocal subgroup with respect to $E_1$.
  By \cite[Theorem 1]{W14},
  $l_{C_N(Q)}(e_\theta)=l$.
  We will claim that $N$ acts transitively on $\mathrm{IBr}(C_N(Q),e_\theta)$.
  Indeed, this holds because
  $l_N(e_\theta)=1$ by the assumption and
  $N/C_N(Q)\cong E_2$ is a cyclic group of order $l$.

  Denote by $\mathrm{Irr}_\mathcal{K}(\tilde{E}_1)_\theta$ the subset
  of $\mathrm{Irr}_\mathcal{K}(\tilde{E}_1)$ consisting of characters covering $\theta$.
  Then $|\mathrm{Irr}_\mathcal{K}(\tilde{E}_1)_\theta|=l$
  and
  we set
  $\mathrm{Irr}_\mathcal{K}(\tilde{E}_1)_\theta=\{\tau_i\,|\, i=1,2,\cdots,l\}$,
  which is transitively acted by $N$.
  Hence,
  we can write
  $\mathrm{Irr}_\mathcal{K}(\tilde{E}_1)_\theta$ as
  $\{\tau^a\,|\,a\in E_2\}$
  for any $\tau\in\mathrm{Irr}_\mathcal{K}(\tilde{E}_1)_\theta$.
  By Clifford theorem,
  we have $\mathrm{Res}_Z^{\tilde{E}_1}(\tau_i)=\theta$ for any $i$
  and
  $\mathrm{Ind}_Z^{\tilde{E}_1}(\theta)
  =\sum\limits_{i=1}^l\tau_i$.
  Let $M$ be a representative of $\tilde{E}_1$-orbit of
  $\mathrm{Irr}_\mathcal{K}(P_1)-\{1_{P_1}\}$,
  where $1_{P_1}$ is the trivial character of $P_1$.
  Then
  $$\mathrm{Irr}_\mathcal{K}(\bar{C}_N(Q),\bar{e}_\theta)=
  \{\tau_i\bar{\zeta}_j\,|\, \bar{\zeta}_j\in\mathrm{Irr}_\mathcal{K}(\bar{P}_2),
     i=1,2,\cdots,l\}\cup
  \{\mathrm{Ind}_{P_1\times Z}^{P_1\rtimes\tilde{E}_1}(\xi\theta)\bar{\zeta}_j\,|\,
     \xi\in M,\bar{\zeta}_j\in\mathrm{Irr}_\mathcal{K}(\bar{P}_2)\}.$$
  We will write $\mathrm{Ind}_{P_1\times Z}^{P_1\rtimes\tilde{E}_1}(\xi\theta)$
  and
  $\bar{\chi}\cdot 1_{\bar{P}_2}$
  as $\mathrm{Ind}(\xi)$
  and $\bar{\chi}$ respectively for simplicity.
  Here,
  $\bar{\chi}$ is an element of
  $\mathcal{CF}_\mathcal{K}(\overline{P_1\rtimes\tilde{E}_1})$.
  Clearly,
  $\mathrm{Ind}(\xi)$ is $N$ and $E_2$-stable
  for any $\xi\in M$.
  Similar to the argument of \cite[Case 2]{W05},
  $$\{(\sum\limits_{i=1}^l\tau_i-\mathrm{Ind}(\xi))\bar{\zeta}\,|\,\xi\in M,
    \bar{\zeta}\in\mathrm{Irr}_\mathcal{K}(\bar{P}_2)\}\cup
    \{\tau_i-\tau_i\bar{\zeta}\,|\,i=1,2,\cdots,l,
         1_{\bar{P}_2}\neq\bar{\zeta}\in\mathrm{Irr}_\mathcal{K}(\bar{P}_2)\}$$
  is a $\mathbb{Z}$-basis of
  $\mathcal{L}_\mathcal{K}^\circ(\bar{C}_N(Q),\bar{e}_\theta)$.

      \vspace{2.5mm}
  {\bf Case 3.1}~~~Assume that $\bar{P}_2=1$, i.e.,
                    $Q=P_2$.

  Set $H=N_G(Q,b_Q)$.
  Then $H=C_G(Q)N_G(P,b_P)$ and
  $b_Q$ is still a block of $H$.
  Let $d$ be the Brauer correspondent of the block $b_Q$ of $H$
  in $N_H(P)$.
  Then $l_{N_H(P)}(d)=1$ by the assumption.
  We claim that $l_H(b_Q)=1$.

  Indeed, considering the canonical map from $\mathcal{O}H$ to $\mathcal{O}(H/Q)$,
  denote by $\bar{X}$ the image of $X$ under this canonical map
  for any subset $X$ of $\mathcal{O}H$.
  Then $\bar{b}_Q$ is still a block of $\bar{C}_G(Q)$
  and $\bar{H}/\bar{C}_G(Q)$ is a cyclic group of order $l$.
  By \cite[Lemma 3.5]{KR},
  $\mathrm{Br}_{\bar{P}}(\bar{b}_Q)=\overline{\mathrm{Br}_P(b_Q)}$.
  Since $l_{N_H(P)}(d)=1$,
  $\bar{d}$ is still a block of $\bar{N}_H(P)$.
  Therefore,
  $\mathrm{Br}_{\bar{P}}(\bar{b}_Q)=\bar{d}$ is a block of $\bar{N}_H(P)$.
 Suppose that the blocks of $\bar{H}$
 covering the block $\bar{b}_Q$ of $\bar{C}_G(Q)$
 have the same defect group $\bar{P}$.
  Then $\bar{b}_Q$ is a block of $\bar{H}$
  since $\mathrm{Br}_{\bar{P}}(\bar{b}_Q)$ is a block of $\bar{N}_H(P)$
  and $N_{\bar{H}}(\bar{P})=\bar{N}_H(P)$.
  Hence,
  it has a defect group $\bar{P}$ which is cyclic by our assumption.
  In particular,
  we have $l_H(b_Q)=l_{\bar{H}}(\bar{b}_Q)=l_{N_H(P)}(d)=1$
  since $N_{\bar{H}}(\bar{P})=\bar{N}_H(P)$.
  Consequently, the argument follows from the lemma below.

\vspace{0.35cm}
\noindent{\it {\bf Lemma 3.4}
  Let $L$ be a normal subgroup of $K$ such that
  $K/L$ is a cyclic $p^\prime$-group.
  Let $i$ be a $K$-stable block of $L$ with defect group $D$.
  For any block $e$ of $K$ covering $i$,
  $e$ has defect group $D$.}

\vspace{0.35cm}
\noindent{\it Proof.}
We will prove it by induction on $K/L$.
Let $M\leq K$ such that $M$ contains $L$
and $|M/L|$ is a prime.
Then $M\trianglelefteq K$ and
$K/M$ is still a cyclic $p^\prime$-group.
Denote by $M[i]$ the subgroup of $M$ consisting of
elements acting on $\mathcal{O}Li$ as inner automorphisms.
Therefore,
$M[i]=M$ or $L$.
Let $f$ be a block of $M$ covered by $e$.
So $f$ covers the block $i$ of $L$.
If $M[i]=M$,
then $\mathcal{O}Mf$ and $\mathcal{O}Li$ are source algebra equivalent
by \cite[Theorem 7]{K90}.
In particular,
the block $f$ has defect group $D$.
If $M[b]=L$,
then $f=i$ by \cite[Theorem 3.5]{D}
and certainly they have the same defect group.
In conclusion,
$D$ is a defect group of the block $f$.
Let $K_f$ be the stabilizer of $f$ in $K$.
Then blocks of $K_f$ covering $f$ have defect group $D$ by induction.
So is $e$.

\vspace{0.35cm}
Moreover, we claim that
there is a regular $E_2$-orbit of $\mathrm{Irr}_\mathcal{K}(\bar{C}_G(Q),\bar{b}_Q)$,
namely,
$H$ acts transitively on it.

Indeed, since the block $\bar{b}_Q$ of $\bar{H}$ has a cyclic defect group,
  it must be nilpotent.
  By \cite[Theorem 3.13]{P11},
  the block $\bar{b}_Q$ of $\bar{C}_G(Q)$ is basic Morita equivalent to
  its Brauer correspondent.
  Note that the block $\bar{b}_Q$ of $\bar{C}_G(Q)$ is not nilpotent
  since $l>1$.
  This implies that every irreducible Brauer character of
  the block $\bar{b}_Q$ of $\bar{C}_G(Q)$
  can be uniquely lifted to an irreducible ordinary character
  by the theory of cyclic blocks.

  On the other hand,
  since $l_{C_G(Q)}(b_Q)=l$ and $l_H(b_Q)=1$ and
  $H/C_G(Q)\cong E_2$ has order $l$,
  $H$ acts transitively on $\mathrm{IBr}(C_G(Q),b_Q)$.
  Combining this with the argument above,
  there exits a regular $H$-orbit of $\mathrm{Irr}_\mathcal{K}(\bar{C}_G(Q),\bar{b}_Q)$.
 We are done.

        \vspace{2.5mm}
  {\bf Case 3.1.1}~~~Assume that $|M|=1$.

  Then $\mathrm{rank}_\mathcal{O}
  (\mathcal{L}_\mathcal{K}^\circ(\bar{C}_N(Q),\bar{e}_\theta))=1$
  and
  $\mathcal{L}_\mathcal{K}^\circ(\bar{C}_N(Q),\bar{e}_\theta)=
  \mathbb{Z}(\mathrm{Ind}(\xi)-\sum\limits_{i=1}^l\tau_i)$.
  Since there is a regular $E_2$-orbit of $\mathrm{Irr}_\mathcal{K}(\bar{C}_G(Q),\bar{b}_Q)$,
  $\mathrm{Irr}_\mathcal{K}(\bar{C}_G(Q),\bar{b}_Q)=\{\chi_0\}\cap
  \{\chi_1,\chi_2,\cdots,\chi_l\}$
  such that
  $\chi_0$ is $E_2$-stable and
  $E_2$ acts regularly on $\{\chi_1,\chi_2,\cdots,\chi_l\}$.
  Then we have
  $$\bar{\Delta}_Q^\circ(\mathrm{Ind}(\xi)-\sum\limits_{i=1}^l\tau_i)=
  \delta_0\chi_0-\sum\limits_{i=1}^l\delta_i\chi_i$$
  for some $\delta_0,\delta_i\in\{\pm1\},i=1,2,\cdots,l$.
  Since $\bar{\Delta}_Q^\circ$ is $E_2$-stable,
  we have $\delta_1=\delta_2=\cdots=\delta_l=\delta_0$.
  If we  write $\{\tau_1,\tau_2,\cdots,\tau_l\}$
  and
  $\{\chi_1,\chi_2,\cdots,\chi_l\}$
  as
  $\{\tau^a\,|\,a\in E_2\}$
  and
  $\{\chi^a\,|\,a\in E_2\}$
  respectively,
  then we can define a bijective isometry as below
  $$\bar{\Delta}_Q:\mathcal{L}_\mathcal{K}(\bar{C}_N(Q),\bar{e}_\theta)\longrightarrow
  \mathcal{L}_\mathcal{K}(\bar{C}_G(Q),\bar{b}_Q)$$
  $$\mathrm{Ind}(\xi)\mapsto\delta_0\chi_0$$
  $$\tau^a\mapsto\delta_0\chi^a.$$
  It is evident that
  it is an extension of $\bar{\Delta}_Q^\circ$ and
  $E_2$-stable.
  We are done for this case.

          \vspace{2.5mm}
  {\bf Case 3.1.2}~~~Assume that $|M|\geq2$.

  Then there are at least two different $\xi_1,\xi_2\in M$.
  So
  $\mathrm{Ind}(\xi_1)-\mathrm{Ind}(\xi_2)\in
  \mathcal{L}_\mathcal{K}^\circ(\bar{C}_N(Q),\bar{e}_\theta)$
  and
  $\langle\mathrm{Ind}(\xi_1)-\mathrm{Ind}(\xi_2),
  \mathrm{Ind}(\xi_1)-\mathrm{Ind}(\xi_2)\rangle=2$.
  Then there exist
  $\chi_1\neq\chi_2\in
  \mathrm{Irr}_\mathcal{K}(\bar{C}_G(Q),\bar{b}_Q)$ such that
  $$\bar{\Delta}_Q^\circ(\mathrm{Ind}(\xi_1)-\mathrm{Ind}(\xi_2))=
  \delta(\chi_1-\chi_2)$$
  for some $\delta\in\{\pm1\}$.
  Since $\bar{\Delta}_Q^\circ$ is $E_2$-stable,
  we have
  ${^a}(\delta\chi_1-\delta\chi_2)=\delta(\chi_1-\chi_2)$
  for any $a\in E_2$.
  This means that
  $\chi_1$ and $\chi_2$ are both $E_2$-stable.

  If there is a $\xi_3\in M$ different from $\xi_1$ and $\xi_2$,
  then there is a $\chi_3\in\mathrm{Irr}_\mathcal{K}(\bar{C}_G(Q),\bar{b}_Q)$
  different from $\chi_1$ and $\chi_2$
  such that
  $$\bar{\Delta}_Q^\circ(\mathrm{Ind}(\xi_1)-\mathrm{Ind}(\xi_3))=
  \delta\chi_1-\delta\chi_3~\mathrm{or}~
  -\delta\chi_2+\delta\chi_3$$
  and
  $\chi_3$ is $E_2$-stable;
  then we may choose the notation in such a way that
  $$\bar{\Delta}_Q^\circ(\mathrm{Ind}(\xi_1)-\mathrm{Ind}(\xi_2))=
  \delta(\chi_1-\chi_2)
  ~\mathrm{and}~
  \bar{\Delta}_Q^\circ(\mathrm{Ind}(\xi_1)-\mathrm{Ind}(\xi_3))=
  \delta(\chi_1-\chi_3)$$
 for some $E_2$-stable elements
  $\chi_1,\chi_2,\chi_3$ of
  $\mathrm{Irr}_\mathcal{K}(\bar{C}_G(Q),\bar{b}_Q).$

  If $|M|\geq 4$,
  then for any
  $\xi\in M-\{\xi_1,\xi_2,\xi_3\}$,
  there is a unique
  $\chi\in\mathrm{Irr}_\mathcal{K}(\bar{C}_G(Q),\bar{b}_Q)-
  \{\chi_1,\chi_2,\chi_3\}$
  such that
  $$\bar{\Delta}_Q^\circ(\mathrm{Ind}(\xi_1)-\mathrm{Ind}(\xi))=
  \delta(\chi_1-\chi)$$
  and
  $\chi$ is $E_2$-stable.

  In conclusion,
  we have an injective isometry
  $$\Phi:\mathbb{Z}\{\mathrm{Ind}(\xi)\,|\,\xi\in M\}\longrightarrow
  \mathcal{L}_\mathcal{K}(\bar{C}_G(Q),\bar{b}_Q)$$
  mapping $\mathrm{Ind}(\xi)$ to
  $\delta\chi_\xi$
  such that
  $$\Phi(\mathrm{Ind}(\xi)-\mathrm{Ind}(\xi^\prime))=
  \bar{\Delta}_Q^\circ(\mathrm{Ind}(\xi)-\mathrm{Ind}(\xi^\prime))$$
  and
  $\chi_\xi$ is $E_2$-stable
  for any
  $\xi,\xi^\prime\in M$.

  Denote
  $\mathrm{Irr}_\mathcal{K}(\bar{C}_G(Q),\bar{b}_Q)-
  \{\chi_\xi\,|\,\xi\in M\}$
  by $\Omega$.
  Then
  $|\Omega|=l$
  and
  $E_2$ acts on $\Omega$.
  Since there is a regular $E_2$-orbit of
  $\mathrm{Irr}_\mathcal{K}(\bar{C}_G(Q),\bar{b}_Q)$,
  $E_2$ acts regularly on $\Omega$.
  This means that
  $\Omega$ can be represented as
  $\{\chi^a\,|\,a\in E_2\}$
  for some $\chi\in\Omega$.

  Now we fix an element $\xi$ of $M$.
  Suppose that $\chi$ does not get involved in
  $\bar{\Delta}_Q^\circ(\mathrm{Ind}(\xi)-\sum\limits_{i=1}^l\tau_i)$.
  Then there is $\xi^\prime\in M$ such that
  $\langle\chi,\bar{\Delta}_Q^\circ(\mathrm{Ind}(\xi^\prime)-
  \sum\limits_{i=1}^l\tau_i)\rangle\neq 0$
  since $\{\mathrm{Ind}(\xi)-\sum\limits_{i=1}^l\tau_i\,|\,
  \xi\in M\}$ is a
  $\mathbb{Z}$-basis of $\mathcal{L}_\mathcal{K}^\circ
  (\bar{C}_N(Q),\bar{e}_\theta)$.
  Hence,
  $\chi$ has to get involved in
  $\bar{\Delta}_Q^\circ(\mathrm{Ind}(\xi)-\sum\limits_{i=1}^l\tau_i)-
  \bar{\Delta}_Q^\circ(\mathrm{Ind}(\xi^\prime)-\sum\limits_{i=1}^l\tau_i)$
  which is $\delta(\chi_\xi-\chi_{\xi^\prime})$.
  This is impossible.
  So $\chi$ must get involved in
   $\bar{\Delta}_Q^\circ(\mathrm{Ind}(\xi)-\sum\limits_{i=1}^l\tau_i)$
   for any $\xi\in M$.
  Since $\bar{\Delta}_Q^\circ$
  and
  $\mathrm{Ind}(\xi)-\sum\limits_{i=1}^l\tau_i$
  are $E_2$-stable,
  $\chi^a$ has to get involved in
   $\bar{\Delta}_Q^\circ(\mathrm{Ind}(\xi)-\sum\limits_{i=1}^l\tau_i)$
   for any $a\in E_2$ and $\xi\in M$.
   Since
   $\langle\mathrm{Ind}(\xi)-\sum\limits_{i=1}^l\tau_i,
   \mathrm{Ind}(\xi)-\sum\limits_{i=1}^l\tau_i\rangle=1+l$
   and
   $\langle\mathrm{Ind}(\xi)-\sum\limits_{i=1}^l\tau_i,
    \mathrm{Ind}(\xi)-\mathrm{Ind}(\xi^\prime)\rangle=1$,
    $\bar{\Delta}_Q^\circ(\mathrm{Ind}(\xi)-\sum\limits_{i=1}^l\tau_i)=
    \delta\chi_\xi-\sum\limits_{a\in E_2}\delta_a\chi^a~\mathrm{or}~
    -\delta\chi_{\xi^\prime}-\sum\limits_{a\in E_2}\delta_a\chi^a$,
    where $\delta_a\in\{\pm1\}$ for any $a\in E_2$.
    Note that the last situation can happen if and only if
    $|M|=2$.
    By switching $\chi_\xi$ and $\chi_{\xi^\prime}$ if necessary,
    we can assume that
    $\bar{\Delta}_Q^\circ(\mathrm{Ind}(\xi)-\sum\limits_{i=1}^l\tau_i)=
    \delta\chi_\xi-\sum\limits_{a\in E_2}\delta_a\chi^a$.
    Since $\bar{\Delta}_Q^\circ$ is $E_2$-stable and
    $E_2$ acts regularly on $\Omega$,
    $\delta_a$ is equal to $\delta$ for any $a\in E_2$.
   Then we can define an $E_2$-stable bijective isometry as follows
  $$\bar{\Delta}_Q:\mathcal{L}_\mathcal{K}(\bar{C}_N(Q),\bar{e}_\theta)\longrightarrow
  \mathcal{L}_\mathcal{K}(\bar{C}_G(Q),\bar{b}_Q)$$
  $$\mathrm{Ind}(\xi)\mapsto\delta\chi_\xi$$
  $$\tau^a\mapsto\delta\chi^a.$$
  It is clear that $\bar{\Delta}_Q$ is an extension of $\bar{\Delta}_Q^\circ$.

   \vspace{2.5mm}
  {\bf Case 3.2}~~~$\bar{P}_2>1$, namely,
  $Q$ is a non-trivial proper subgroup of $P_2$.

  Then $\mathrm{Ind}(\xi)-\mathrm{Ind}(\xi)\bar{\zeta}\in
  \mathcal{L}_\mathcal{K}^\circ(\bar{C}_N(Q),\bar{e}_\theta)$
  for any $\xi\in M$ and
  $1_{\bar{P}_2}\neq\bar{\zeta}\in
  \mathrm{Irr}_\mathcal{K}(\bar{P}_2)$.

  Now we fix an element $\xi\in M$.
  Since $p$ is odd,
  $|\bar{P}_2|\geq 3$.
  Then there are at least two elements $\bar{\zeta}$ and $\bar{\zeta}^\prime$
  of $\mathrm{Irr}_\mathcal{K}(\bar{P}_2)$
  different from $1_{\bar{P}_2}$.
  With the same argument in the first three paragraphs in Case 3.1.2,
  we can get a subset $\{\chi_\xi,\chi_{\bar{\zeta}}\,|\,
  1_{\bar{P}_2}\neq\bar{\zeta}\in\mathrm{Irr}_\mathcal{K}(\bar{P}_2)\}$ of
  $\mathrm{Irr}_\mathcal{K}(\bar{C}_G(Q),\bar{b}_Q)$
  such that
  $$\bar{\Delta}_Q^\circ(\mathrm{Ind}(\xi)-\mathrm{Ind}(\xi)\bar{\zeta})=
  \delta(\chi_\xi-\chi_{\bar{\zeta}})$$
  for any $1_{\bar{P}_2}\neq\bar{\zeta}\in\mathrm{Irr}_\mathcal{K}(\bar{P}_2)$,
  where $\delta\in\{\pm1\}$.

  Given any $1\neq a\in E_2$ and
  $1_{\bar{P}_2}\neq\bar{\zeta}\in\mathrm{Irr}_\mathcal{K}(\bar{P}_2)$,
  ${^a}(\mathrm{Ind}(\xi)\bar{\zeta})=\mathrm{Ind}(\xi)({^a}\bar{\zeta})$.
  Since $\bar{\Delta}_Q^\circ$ is $E_2$-stable,
  this means
  ${^a}\chi_\xi-{^a}\chi_{\bar{\zeta}}=\chi_\xi-\chi_{{^a}\bar{\zeta}}$.
  Hence,
  we have
  $\chi_\xi$ is $E_2$-stable
  and
  ${^a}\chi_{\bar{\zeta}}=\chi_{{^a}\bar{\zeta}}$.
  On the other hand,
  $$(\mathrm{Ind}(\xi)-\mathrm{Ind}(\xi)\bar{\zeta})\bar{\zeta}=
  (\mathrm{Ind}(\xi)-\mathrm{Ind}(\xi)\bar{\zeta}^2)-
  (\mathrm{Ind}(\xi)-\mathrm{Ind}(\xi)\bar{\zeta}).$$
  Since $\bar{\Delta}_Q^\circ$ is compatible with $\ast$-structure,
  using $\bar{\Delta}_Q^\circ$ on both sides in the above equality,
  we can get
  $$\delta(\chi_\xi-\chi_{\bar{\zeta}})\ast\bar{\zeta}=
  \delta(\chi_{\bar{\zeta}}-\chi_{\bar{\zeta}^2}).$$
  Therefore,
  $\chi_{\bar{\zeta}}=\chi_\xi\ast\bar{\zeta}$ for any
  $1_{\bar{P}_2}\neq\bar{\zeta}\in\mathrm{Irr}_\mathcal{K}(\bar{P}_2)$.

  Suppose that there is another element $\xi^\prime$ of $M$
  different from $\xi$.
  Similarly, we can get a subset
  $\{\chi_{\xi^\prime}\ast\bar{\zeta}\,|\,
  \bar{\zeta}\in\mathrm{Irr}_\mathcal{K}(\bar{P}_2)\}$
  of $\mathrm{Irr}_\mathcal{K}(\bar{C}_G(Q),\bar{b}_Q)$
  such that
  $\bar{\Delta}_Q^\circ(\mathrm{Ind}(\xi^\prime)-
  \mathrm{Ind}(\xi^\prime)\bar{\zeta})=
  \delta^\prime(\chi_{\xi^\prime}-\chi_{\xi^\prime}\ast\bar{\zeta})$
  for any $1_{\bar{P}_2}\neq\bar{\zeta}\in\mathrm{Irr}_\mathcal{K}(\bar{P}_2)$
  and
  $\chi_{\xi^\prime}$ is $E_2$-stable,
  where $\delta^\prime\in\{\pm 1\}$.
  Assume that
  $\{\chi_\xi\ast\bar{\zeta}\,|\,
  \bar{\zeta}\in\mathrm{Irr}_\mathcal{K}(\bar{P}_2)\}\cap
  \{\chi_{\xi^\prime}\ast\bar{\zeta}\,|\,
  \bar{\zeta}\in\mathrm{Irr}_\mathcal{K}(\bar{P}_2)\}\neq\emptyset$.
  Then there is $\bar{\zeta}_0\in\mathrm{Irr}_\mathcal{K}(\bar{P}_2)$
  such that
  $\chi_\xi=\chi_{\xi^\prime}\ast\bar{\zeta}$.
  If $\bar{\zeta}_0=1_{\bar{P}_2}$,
  then $\chi_\xi=\chi_{\xi^\prime}$.
  This implies that
  $\mathrm{Ind}(\xi)-\mathrm{Ind}(\xi)\bar{\zeta}=
  \pm(\mathrm{Ind}(\xi^\prime)-\mathrm{Ind}(\xi^\prime)\bar{\zeta})$
  for any
  $1_{\bar{P}_2}\neq\bar{\zeta}\in\mathrm{Irr}_\mathcal{K}(\bar{P}_2)$.
  This is impossible.
  Then $\bar{\zeta}_0$ is non-trivial.
  But it implies that $\chi_{\xi^\prime}\ast\bar{\zeta}_0^2=\chi_{\xi^\prime}$
  since
  $\langle\mathrm{Ind}(\xi)-\mathrm{Ind}(\xi)\bar{\zeta}_0,
  \mathrm{Ind}(\xi^\prime)-\mathrm{Ind}(\xi^\prime)\bar{\zeta}_0\rangle=0$.
It is well-known that
$\mathrm{Irr}_\mathcal{K}(\bar{P}_2)\backslash\{1_{\bar{P}_2}\}$ acts freely on
irreducible ordinary characters of height zero in the block $\bar{b}_Q$ of $\bar{C}_G(Q)$
(see \cite[\S 1]{R}).
Hence,
$\bar{\zeta}_0^2=1_{\bar{P}_2}$
since the defect group of the block $\bar{b}_Q$ of $\bar{C}_G(Q)$ is cyclic.
But it is impossible because $p$ is odd.
Then
 $$\{\chi_\xi\ast\bar{\zeta}\,|\,
  \bar{\zeta}\in\mathrm{Irr}_\mathcal{K}(\bar{P}_2)\}\cap
  \{\chi_{\xi^\prime}\ast\bar{\zeta}\,|\,
  \bar{\zeta}\in\mathrm{Irr}_\mathcal{K}(\bar{P}_2)\}=\emptyset$$
  for any different $\xi,\xi^\prime\in M$.
 It is clear that $\chi_\xi\ast\bar{\zeta}$ is an irreducible ordinary character
 in the block $\bar{b}_Q$ of $\bar{C}_G(Q)$
 by \cite[Corollary]{BP}.
 Then we get an injective isometry
  $$\Psi:\mathbb{Z}\{\mathrm{Ind}(\xi)\bar{\zeta}\,|\,
  \xi\in M,\bar{\zeta}\in\mathrm{Irr}_\mathcal{K}(\bar{P}_2)\}
  \longrightarrow
  \mathcal{L}_\mathcal{K}(\bar{C}_G(Q),\bar{b}_Q)$$
  mapping $\mathrm{Ind}(\xi)\bar{\zeta}$ to
  $\delta_\xi(\chi_\xi\ast\bar{\zeta})$
  such that
  $\Psi(\mathrm{Ind}(\xi)-\mathrm{Ind}(\xi)\bar{\zeta})=
  \bar{\Delta}_Q^\circ(\mathrm{Ind}(\xi)-\mathrm{Ind}(\xi)\bar{\zeta})$
  and
  $\chi_\xi$ is $E_2$-stable
  for any $\xi\in M$ and
  $1_{\bar{P}_2}\neq\bar{\zeta}\in\mathrm{Irr}_\mathcal{K}(\bar{P}_2)$,
  where $\delta_\xi\in\{\pm 1\}$.

  At the same time,
  $\tau-\tau\bar{\zeta}\in
  \mathcal{L}_\mathcal{K}^\circ(\bar{C}_N(Q),\bar{e}_\theta)$
  for any $\tau\in\mathrm{Irr}_\mathcal{K}(\tilde{E}_1)_\theta$
  and
  $1_{\bar{P}_2}\neq\bar{\zeta}\in\mathrm{Irr}_\mathcal{K}(\bar{P}_2)$.
Take an element $\tau$ of $\mathrm{Irr}_\mathcal{K}(\tilde{E}_1)_\theta$.
  With the same arguments as above,
  we can get an element $\chi_\tau$ of
  $\mathrm{Irr}_\mathcal{K}(\bar{C}_G(Q),\bar{b}_Q)$
  and
  $\delta_\tau\in\{\pm1\}$
  such that
  $\bar{\Delta}_Q^\circ(\tau-\tau\bar{\zeta})=
  \delta_\tau(\chi_\tau-\chi_\tau\ast\bar{\zeta})$
  for any
  $1_{\bar{P}_2}\neq\bar{\zeta}\in\mathrm{Irr}_\mathcal{K}(\bar{P}_2)$.
  Choosing any $1\neq a\in E_2$ and
  $1_{\bar{P}_2}\neq\bar{\zeta}\in\mathrm{Irr}_\mathcal{K}(\bar{P}_2)$,
  since $\bar{\Delta}_Q^\circ$ is $E_2$-stable, we have
  $$\delta_{{^a}\tau}(\chi_{{^a}\tau}-\chi_{{^a}\tau}\ast{^a}\bar{\zeta})=
  \bar{\Delta}_Q^\circ({^a}\tau-{^a}\tau({^a}\bar{\zeta}))=
  {^a}(\bar{\Delta}_Q^\circ(\tau-\tau\bar{\zeta}))=
  \delta_\tau({^a}\chi_\tau-{^a}\chi_\tau\ast{^a}\bar{\zeta}).$$
  Then $\chi_{{^a}\tau}={^a}\tau$ or
  $\chi_{{^a}\tau}={^a}\chi_\tau\ast{^a}\bar{\zeta}$.
  If $\chi_{{^a}\tau}={^a}\chi_\tau\ast{^a}\bar{\zeta}$,
  then
  ${^a}\chi_\tau=\chi_{{^a}\tau}\ast{^a}\bar{\zeta}$.
  Therefore,
  $\chi_{{^a}\tau}=\chi_{{^a}\tau}\ast{^a}(\bar{\zeta}^2)$,
  which is impossible.
  Hence,
  $\chi_{{^a}\tau}={^a}\chi_\tau$
  and
  $\delta_{{^a}\tau}=\delta_\tau$
  for any $a\in E_2$
  since $E_2$ acts transitively on $\mathrm{Irr}_\mathcal{K}(\tilde{E}_1)_\theta$.
  And we denote $\delta_\tau$ by $\delta$.
  By the facts that
  $\langle\tau-\tau\bar{\zeta},
   \tau^\prime-\tau^\prime\bar{\zeta}^\prime\rangle=0$
   and
  $\langle\mathrm{Ind}(\xi)-\mathrm{Ind}(\xi)\bar{\zeta},
  \tau-\tau\bar{\zeta}^\prime\rangle=0$
  for any
  $\tau\neq\tau^\prime\in\mathrm{Irr}_\mathcal{K}(\tilde{E}_1)_\theta$
  and
  $\xi\in M$
  and
  $\bar{\zeta},\bar{\zeta}^\prime\in
  \mathrm{Irr}_\mathcal{K}(\bar{P}_2)-\{1_{\bar{P}_2}\}$,
  we can get
  $$\{\chi_\tau\ast\bar{\zeta}\,|\,
  \bar{\zeta}\in\mathrm{Irr}_\mathcal{K}(\bar{P}_2)\}\cap
  \{\chi_{\tau^\prime}\ast\bar{\zeta}\,|\,
  \bar{\zeta}\in\mathrm{Irr}_\mathcal{K}(\bar{P}_2)\}=\emptyset$$
  and
  $$\{\chi_\xi\ast\bar{\zeta}\,|\,
  \xi\in M, \bar{\zeta}\in\mathrm{Irr}_\mathcal{K}(\bar{P}_2)\}\cap
  \{\chi_\tau\ast\bar{\zeta}\,|\,
  \tau\in\mathrm{Irr}_\mathcal{K}(\tilde{E}_1)_\theta,
  \bar{\zeta}\in\mathrm{Irr}_\mathcal{K}(\bar{P}_2)\}=\emptyset.$$
  Hence,
  we have a well-defined $E_2$-stable bijective isometry as below
  $$\bar{\Delta}_Q:\mathcal{L}_\mathcal{K}(\bar{C}_N(Q),\bar{e}_\theta)\longrightarrow
  \mathcal{L}_\mathcal{K}(\bar{C}_G(Q),\bar{b}_Q)$$
  $$\mathrm{Ind}(\xi)\bar{\zeta}\mapsto\delta_\xi\chi_\xi\ast\bar{\zeta}$$
  $${^a}\tau\bar{\zeta}\mapsto\delta{^a}\chi_\tau\ast\bar{\zeta}.$$
  It suffices to show that $\bar{\Delta}_Q$ is an extension of $\bar{\Delta}_Q^\circ$,
  namely,
  $$\bar{\Delta}_Q^\circ(\mathrm{Ind}(\xi)-\sum\limits_{i=1}^l\tau_i)=
  \delta_\xi\chi_\xi-\delta\sum\limits_{i=1}^l\chi_{\tau_i}$$
  for any $\xi\in M$.

  Choose an element $\xi$ of $M$.
  Since
  $\langle\mathrm{Ind}(\xi)-\sum\limits_{i=1}^l\tau_i,
  \tau-\tau\bar{\zeta}\rangle=-1$,
  then at least $\chi_\tau$ and $\chi_\tau\ast\bar{\zeta}$ must get involved in
  $\bar{\Delta}_Q^\circ(\mathrm{Ind}(\xi)-\sum\limits_{i=1}^l\tau_i)$
  for any $\tau\in\mathrm{Irr}_\mathcal{K}(\tilde{E}_1)_\theta$
  and
  $\bar{\zeta}\in\mathrm{Irr}_\mathcal{K}(\bar{P}_2)-\{1_{\bar{P}_2}\}$.

  Keep the notation as above.
  Suppose that there are $\tau$ and $\bar{\zeta}$ such that
  $\chi_\tau\ast\bar{\zeta}$ gets involved in
  $\bar{\Delta}_Q^\circ(\mathrm{Ind}(\xi)-\sum\limits_{i=1}^l\tau_i)$.
  Since
  $\langle\mathrm{Ind}(\xi)-\sum\limits_{i=1}^l\tau_i,
  \tau\bar{\zeta}-\tau\bar{\zeta}^\prime\rangle=0$
  for any $\bar{\zeta}^\prime\in\mathrm{Irr}_\mathcal{K}(\bar{P}_2)$
  different from $\bar{\zeta}$ and $1_{\bar{P}_2}$,
  $\chi_\tau\ast\bar{\zeta}$ must get involved in
  $\bar{\Delta}_Q^\circ(\mathrm{Ind}(\xi)-\sum\limits_{i=1}^l\tau_i)$
  for any $\bar{\zeta}$.
  At the same time,
  since $\bar{\Delta}_Q^\circ$ is $E_2$-stable
  and
  $\mathrm{Ind}(\xi)-\sum\limits_{i=1}^l\tau_i$ is $E_2$-stable,
  we have
  ${^a}(\chi_\tau\ast\bar{\zeta})$ must get involved in
  $\bar{\Delta}_Q^\circ(\mathrm{Ind}(\xi)-\sum\limits_{i=1}^l\tau_i)$
  for any $a\in E_2$ and $\bar{\zeta}$.
  Then there are at least
  $l\cdot(|\bar{P}_2|-1)$ different irreducible characters involved in
  $\bar{\Delta}_Q^\circ(\mathrm{Ind}(\xi)-\sum\limits_{i=1}^l\tau_i)$.
  This is impossible since
  $\langle\mathrm{Ind}(\xi)-\sum\limits_{i=1}^l\tau_i,
  \mathrm{Ind}(\xi)-\sum\limits_{i=1}^l\tau_i\rangle=1+l$
  and $|\bar{P}_2|-1\geq 2$ and $l>1$.

  So for any $\xi\in M$,
  $\bar{\Delta}_Q^\circ(\mathrm{Ind}(\xi)-\sum\limits_{i=1}^l\tau_i)=
  a_\chi\chi-\delta\sum\limits_\tau\chi_\tau$.
  Here,
  $a_\chi\in\{\pm 1\}$ and
  $\chi$ is an element of
  $\mathrm{Irr}_\mathcal{K}(\bar{C}_G(Q),\bar{b}_Q)-
  \{\chi_\tau\,|\,
  \tau\in\mathrm{Irr}_\mathcal{K}(\tilde{E}_1)_\theta\}$.
  Since
  $\langle\mathrm{Ind}(\xi)-\sum\limits_{i=1}^l\tau_i,
  \mathrm{Ind}(\xi)-\mathrm{Ind}(\xi)\bar{\zeta}\rangle=1$
  for any $\bar{\zeta}\neq 1_{\bar{P}_2}$,
  we have
  $a_\chi=\delta_\xi$
  and
  $\chi=\chi_\xi$.
  We are done.
  \end{proof}

  Then the proof of Theorem \ref{MT} will follow by
  Theorem \ref{MT'} and \cite[3.4.2]{PU}.

\end{document}